\documentclass[a4ppaper,10pt]{amsart}

\usepackage{enumerate}
\usepackage[latin1]{inputenc}
\usepackage{amsmath}
\usepackage{amsthm}
\usepackage{latexsym}
\usepackage{amssymb}
\usepackage{color}

\usepackage[all]{xy}
\usepackage{calc}
\usepackage{multicol}
\newtheorem{thm}{Theorem}[section]
\newtheorem{prop}{Proposition}[section]

\newtheorem{definition}{Definition}[section]
\newtheorem{lem}{Lemma}[section]

\newtheorem{remark}{Remark}[section]

\newcommand{\R}{\mathbb{R}}
\numberwithin{equation}{section}
\newcommand{\N}{\mathbb{N}}

\begin{document}

\title[Instability for 4NLS]{Strong instability of ground states to a fourth order Schr\" odinger equation}

\thanks{D. Bonheure \& J.B. Casteras are supported by INRIA - Team MEPHYSTO, MIS F.4508.14 (FNRS), PDR T.1110.14F (FNRS); J.B. Casteras is supported by the
Belgian Fonds de la Recherche Scientifique -- FNRS;
D. Bonheure is partially supported by the project ERC Advanced Grant  2013 n. 339958: ``Complex Patterns for Strongly Interacting Dynamical Systems - COMPAT'' and by ARC AUWB-2012-12/17-ULB1- IAPAS, This work has been carried out in the framework of the project NONLOCAL (ANR-14-CE25-0013),
funded by the French National Research Agency (ANR)}
%\thanks{We thank Enno Lenzmann for useful comments on our first version of this manuscript.}

\author[Bonheure, Casteras, Gou and Jeanjean]{Denis Bonheure \and Jean-Baptiste Casteras \and Tianxiang Gou \and Louis Jeanjean}

\address{Denis Bonheure, Jean-Baptiste Casteras
\newline \indent D\'epartement de Math\'ematiques, Universit\'e Libre de Bruxelles,
\newline \indent CP 214, Boulevard du triomphe, B-1050 Bruxelles, Belgium,
\newline \indent and INRIA- team MEPHYSTO.}
\email{Denis.Bonheure@ulb.ac.be}
\email{jeanbaptiste.casteras@gmail.com}

\address{Tianxiang Gou
\newline \indent Laboratoire de Math\' ematiques (UMR 6623), Universit\' e Bourgogne Franche-Comt\'e,
\newline \indent 16, Route de Gray 25030 Besan\c con Cedex, France. 
\newline \indent and
\newline \indent School of Mathematics and Statistics,
\newline \indent Lanzhou University, Lanzhou, Gansu 730000, People's Republic of China.}
\email{tianxiang.gou@univ-fcomte.fr}

\address{Louis Jeanjean
\newline \indent Laboratoire de Math\' ematiques (UMR 6623), Universit\' e Bourgogne Franche-Comt\'e.
\newline \indent 16, Route de Gray 25030 Besan\c con Cedex, France.}
\email{louis.jeanjean@univ-fcomte.fr}

%\author[Bonheure, Casteras, Moreira dos Santos and Nascimento]{Denis Bonheure \and  \and Ederson Moreira dos Santos \and Robson Nascimento}

\begin{abstract}
In this note we prove the instability by blow-up of the ground state solutions for a class of fourth order Schr\" odinger equations. This extends the first rigorous results on blowing-up solutions for the biharmonic NLS due to Boulenger and Lenzmann \cite{BoLe} and confirm numerical conjectures from \cite{BaFi, BaFiMa1, BaFiMa, FiIlPa}.

\end{abstract}
\subjclass[2010]{35Q55,35B35,35J50,35Q60,35Q40}

\keywords{Fourth order Schr\"odinger equation ; strong instability ; ground state solution ; blowing up solutions.}

\maketitle

\section{Introduction}

In this note we are concerned with the following biharmonic NLS equation
\begin{equation}
\label{4nlsdis}
\begin{cases}
i \partial_t  \psi-\gamma \Delta^2 \psi + \mu\Delta \psi +|\psi|^{2\sigma} \psi =0,\\
\psi (0,x)=u_0 (x) \in H^2 (\R^N).
\end{cases}
\end{equation}
Here $\gamma >0, \, \mu \geq 0$ are given parameters and $0 < N \sigma < 4^*$, where we agree that
$4^* := \frac{4N}{(N-4)^+},$ namely $ 4^*= \infty \mbox{ if } N\leq 4,$ and $4^* = \frac{4N}{N-4} \mbox{ if } N \geq 5.$
By $H^2 (\R^N)$ we mean $H^2 (\R^N;\mathbb C)$, so that the solution is a complex valued function.  For $1 \leq q \leq \infty,$ we denote by $L^q(\R^N)$  the usual Lebesque space with norm $||u||_q^q:= \int_{\R^N}|u|^q \, dx. $ The space $H^2 (\R^N)$ is equipped with its standard norm.
% $\|u\|_{H^2}^2:=\|\Delta u\|_2^2 + \|\nabla u\|_2^2 + \|u\|_2^2$.

\medskip

We recall that standing waves, or so-called waveguides in optics, to \eqref{4nlsdis} are solutions of the form 
$$\psi (t,x)=e^{i\omega t} u(x),\quad \omega \in \R.$$ 
The function $u$ then satisfies the following elliptic equation
\begin{equation}
\label{4nls}
\gamma \Delta^2 u - \mu \Delta u+\omega u=|u|^{2\sigma}u,\quad  u \in H^2(\R^N).
\end{equation}
It is well known that NLS, namely \eqref{4nlsdis} with $\gamma=0$ and $\mu=1$, can become singular at finite time, see for instance \cite{FiIlPa} and the classical references therein. Karpman and Shagalov \cite{KaSh} were apparently the first to study the regularization and stabilization effect of a small fourth-order dispersion. One of their results shows, by help of some stability analysis and numerical computations, that when $N\sigma\le 2$, the standing waves are stable for all $\gamma$, when $2<N\sigma<4$ they are stable for small values of $\gamma$ and that they are unstable when $\sigma N >4$.  Adding a small fourth-order dispersion term thus leads to a new critical exponent/dimension. In particular, the Kerr nonlinearity $\sigma=1$ becomes subcritical in dimension $2$ and $3$ which is obviously an important feature of this extended model e.g. for its relevance in optics. 

In \cite{FiIlPa}, Fibich et al. also motivated the study of \eqref{4nlsdis} by recalling that NLS arises from NLH (the nonlinear Helmholtz equation) as a paraxial approximation. Since NLS can become singular at a finite time, this suggests that some of the small terms neglected in the approximation, plays in fact a role to prevent the blow up. Fibich et al. addressed naturally the question whether nonparaxiality prevents the collapse as the small fourth-order dispersion coefficient $\gamma$ is shown to be part of the nonparaxial correction to NLS. 

In \cite{FiIlPa}, global existence in time is shown by applying the arguments of Weinstein \cite{We}. The role of the new critical exponent $\sigma =4/N$ with respect to global existence is discussed.  The necessary Strichartz estimates were shown by Ben-Artzi et al. \cite{BAKoSa}.

One central question that arises concerning the dynamics of \eqref{4nlsdis} is the stability of its ground state standing waves. A standing wave $e^{i\omega t} u(x)$ is called a ground state if $u$ minimizes 
(among all standing wave solutions with the same frequency $\omega$) the action functional 
 $E_w:H^2(\R^N)\to \R$ defined by 
\begin{align*}
E_w (u)&=  \frac{\gamma}{2}\|\Delta u\|_2^2+ \frac{\mu}{2} \|\nabla u\|_2^2+ \frac{\omega}{2} \|u\|_2^2 - \frac{1}{2\sigma+2 }\|u\|_{2\sigma +2}^{2\sigma+2}\\
&: = E_0(u) + \frac{\omega}{2} \| u\|_2^2.
\end{align*}
The ground state energy is thus given by
\begin{equation}
\label{enerGS}
d_{\omega}:= \inf \{E_{\omega}(u) \ | \ u\in H^2(\R^N) \backslash \{0\},\ E_{\omega}^\prime (u)=0  \}.
\end{equation}
The ground state level $d_{\omega}$ is known to be reached under general assumptions \cite{BoCaSaNa, BoNa}. It also enjoys various alternative characterizations that we comment below. \medskip

In the case $0 < \sigma N < 4$, which is referred to as the mass-subcritical case, the global existence for the Cauchy problem holds, see \cite{FiIlPa,Pa}  and it is conjectured that ground state solutions are orbitally stable. This is proved in \cite{BoCaSaNa} (see also \cite{NaPa}) under additional assumptions among which the fact that they are non degenerate, see \cite{BoCaSaNa} for a precise statement. \medskip

Here we focus on the mass-critical case $\sigma N =4$ and the mass super-critical case $4 < \sigma N < 4^*$. In this range several numerical studies \cite{BaFi, BaFiMa1, BaFiMa} led to conjecture the existence of blowing-up solutions but the first analytical results were only recently obtained in \cite{BoLe}. In \cite{BoLe} however except in the case where $\mu =0$ and for a mass-supercritical nonlinearity it is assumed that $E_0(u_0) <0$ where $u_0$ is the initial datum.  We refer to Theorems 1 and Theorem 3 of \cite{BoLe} for precise statements.  As we shall see any ground state $u \in H^2(\R^N)$ to \eqref{4nls}  do satisfy $E_0(u) >0$ and thus the main results of \cite{BoLe} are not directly applicable.\medskip

In this note, under general assumptions, we prove that if a ground state is radially symmetric then it is unstable by blow-up in finite time. 

\begin{definition}\label{unstable}
We say that a solution $ u \in H^2(\R^N)$ to \eqref{4nls} is unstable by blow-up in finite (respectively infinite) time if, for all $\varepsilon>0$, there exists $v\in H^2 (\R^N)$ such that 
$\|v - u \|_{H^2}<\varepsilon$ and the solution $\phi(t)$ to \eqref{4nlsdis} with initial data $\phi (0)=v$ blows up in finite (respectively infinite) time in the $H^2$ norm.
\end{definition}

We now state our  main results.

\begin{thm}
\label{mainthm1}
Suppose $ N \geq 2$, $\gamma >0$, $\omega>0$ and $\sigma \leq 4$. If $\mu > 0$ and $4\leq \sigma N < 4^*$ or $\mu = 0$ and $4 < \sigma N < 4^*$, then any radial ground state is unstable by blow-up in finite time. 
\end{thm}

We can only prove a weaker statement when $\sigma>4$ and in the critical case $\mu = 0$ and $\sigma N=4$. 

\begin{thm}
\label{mainthm11}
Suppose $\gamma >0$ and $\omega>0$. Assume either 
$\mu = 0$, $\sigma N=4$ and $ N \geq 2$, 
or $N\in\{2,3,4\}$, $\sigma > 4$ and $\mu \ge 0$. Then any radial ground state is unstable by blow-up in finite or infinite time.
\end{thm}

\begin{remark}\label{rq1}
We do not require in Theorems \ref{mainthm1} and \ref{mainthm11} that all ground state solutions are radially symmetric. Indeed, as it will be clear, the arguments of our proofs remain unchanged if one replaces $H^2(\R^N)$ by its subspace $H_{rad}^2(\R^N)$ of radially symmetric functions. Note also that our results apply to the solutions which are ground states of $E_{\omega}$ considered on $H_{rad}^2(\R^N)$.
\end{remark}

This restriction to radial data originated from the use of key results from \cite{BoLe}. To handle a ground state we thus need to know if it is radially symmetric. \medskip

We now recall the results of existence of radial  ground state solutions to \eqref{4nls} from \cite{BoCaSaNa, BoLe}.

\begin{thm}
\label{thm1.1bonnas}
Assume $\gamma >0,\, \mu \geq 0, \, \omega>0$ and $ 0 <\sigma N< 4^*$. Then \eqref{4nls} has a ground state solution.  In addition
\begin{enumerate}
\item If $ \mu \geq 2 \sqrt{\gamma\omega}$, then any ground state solution  is radially symmetric around some point.
\item If $\sigma \in \N$ then at least one ground state solution is radially symmetric around some point.
\end{enumerate}
\end{thm}

Except for the statement in the case where $\sigma \in \N$ Theorem \ref{thm1.1bonnas} is \cite[Theorem 3.3]{BoCaSaNa}, see also \cite[Theorem 1.1]{BoNa}. The result when $\sigma \in \N$ follows immediately combining the argument of the proof of {\cite[Theorem 1.2]{BoCaSaNa}} with the result of \cite[Appendix A]{BoLe} on the Fourier symmetrization, see \cite[Proposition 3.1]{BoCaSaNa}.

Taking into account Theorem \ref{thm1.1bonnas}, we obtain the following corollary of Theorem \ref{mainthm1}.

\begin{thm}
\label{mainthm2}
Suppose $N \geq 2, \, \gamma >0,\, \mu \geq 0, \, \omega>0$ and $4\leq \sigma N < 4^*$. Then \eqref{4nls} admits ground state solutions and 
\begin{enumerate}
\item If $ \mu \geq 2\sqrt{\gamma\omega}$ and $\sigma \leq 4$, any ground state solution is unstable  by blow-up in finite time. 
\item If $ \mu \geq 2\sqrt{\gamma\omega}$ and $\sigma > 4$, any ground state solution is unstable by blow-up in finite or infinite time. 
\item If $\sigma \in \N$ there exists at least one ground state solution which is unstable by blow-up in finite time (finite or infinite time if either $\mu=0$ and $\sigma N=4$ or $N\in\{2,3,4\}$, $\sigma > 4$ and $\mu \ge 0$). 

%\item If $\sigma \in \N$ there exists at least one ground state solution which is unstable by blow-up in finite time (finite or infinite time if either $\mu=0$ and $\sigma N=4$ or $N\in\{2,3,4\}$, $\sigma > 4$ and $\mu \ge 0$). 
\end{enumerate}
\end{thm}

%\begin{remark}\label{rq2}
%In \cite{BoLe} it is shown, under the assumptions of Theorem \ref{mainthm1}, that if $\sigma \in \N$ there exists at least one radially symmetric ground state.  Thus in view of Remark %\ref{rq1} we obtain the instability by blow-up of this ground state without having to assume that $\mu \geq \sqrt{\gamma \omega}$.
%\end{remark}

Our extension, Theorem \ref{mainthm1}, of the results of \cite{BoLe} relies in particular on an appropriate variational characterization of the ground states. Actually we adapt here the approach of \cite{CoJeSq} which itself was based on \cite{Le} where a new light on the classical approach of Berestycki-Cazenave \cite{BeCa} was given. \medskip

The note is organized as follows. In Section \ref{sec:prelim} we first derive some identities satisfied by any solution to (\ref{4nls}) and then we establish an alternative variational characterization of the ground states. In the same section, we recall some facts on the local existence for the Cauchy problem and a blow-up alternative. The proofs of Theorems \ref{mainthm1} and \ref{mainthm11} are given in Section \ref{sec:proof}. \medskip

{\bf Acknowlegments} We thank Enno Lenzmann for useful comments on our first version of this manuscript.

\section{Preliminary results.}\label{sec:prelim}

\subsection{Some identities satisfied by the solutions to (\ref{4nls})} 
To begin with, we derive some classical identities satisfied by the solutions to \eqref{4nls}. They permit to show, in particular, that our ground states satisfy $E_0(u) >0$. We include the proofs for completeness. 

\begin{lem}\label{identities}
If $u \in H^2(\R^N)$ is a solution to (\ref{4nls}), then it satisfies $I_w(u)=P_w(u)=Q(u)=0$ where
$$I_w(u)=\gamma\|\Delta u\|_2^2+\mu\|\nabla u\|_2^2+\omega \|u\|_2^2  -\|u\|_{2\sigma +2}^{2\sigma+2}. $$
$$
P_w(u):=\frac{(N-4)\gamma}{2}\|\Delta u\|_2^2+ \frac{(N-2)\mu}{2}\|\nabla u\|_2^2+ \frac{N \omega}{2} \|u\|_2^2  - \frac{N}{2 \sigma +2}\|u\|_{2\sigma +2}^{2\sigma+2},
$$
and
$$Q(u)=\gamma\|\Delta u\|_2^2+\frac{\mu}{2}\|\nabla u\|_2^2-\frac{\sigma N}{2(2\sigma+2)}\|u\|_{2\sigma +2}^{2\sigma+2}.$$
\end{lem}

\begin{proof}
Since $u \in H^2(\R^N)$ is a solution to \eqref{4nls}, multiplying (\ref{4nls}) by $u$ and integrating we get that $I_{\omega}(u)=0$. Next, we notice that $Q(u)= \frac{N}{4}I_{\omega}(u) - \frac{1}{2}P_{\omega}(u)$. Therefore to prove that $Q(u)=0$, we only need to show that $P_{\omega}(u)=0$. This last identity is usually referred to as a Derrick-Pohozaev identity. To establish it we closely follow the proof of \cite[Proposition 1]{BeLi1}. First multiplying \eqref{4nls} by
$x \cdot \nabla u$ and integrating on $B_R(0)$ for some $R >1$, we have 
\begin{align} \label{inte}
 \int_{B_R(0)}  \gamma (x \cdot \nabla u)  \Delta^2 u- \mu  (x \cdot \nabla u) \Delta u
+ \omega(x \cdot \nabla u)  u \, dx =  \int_{B_R(0)} (x \cdot \nabla u) |u|^{2 \sigma} u \, dx.
\end{align}
In a first time, we focus on the first left-hand side term of \eqref{inte}. Integrating by parts, we find
\begin{align*}
\gamma \int_{B_R(0)}   (x \cdot \nabla u)  \Delta^2 u \, dx  &=  - \gamma \int_{B_R(0)}  \nabla(x \cdot \nabla u) \cdot \nabla (\Delta u) \, dx \\
&\ \ +  \gamma \int_{\partial B_R(0)} (\nabla (\Delta u) \cdot {\bf {n}})(x\cdot \nabla u) \, dS \\
&=\gamma \int_{B_R(0)} \Delta(x \cdot \nabla u) \Delta u \, dx  \\
&\ \ - \gamma \int_{\partial B_R(0)} (\nabla (x \cdot \nabla u) \cdot {\bf {n}}) \Delta u  -(\nabla (\Delta u) \cdot {\bf {n}})(x\cdot \nabla u) \, dS,
\end{align*}
where ${\bf {n}}:={\bf {n}}_{x}=\frac{x}{R}$ denotes the unit outward normal at $ x \in \partial{B_R(0)}$.
Integrating by parts one more time, we have
\begin{align*}
\gamma \int_{B_R(0)} \Delta(x \cdot \nabla u)  \Delta u \, dx
&= 2 \gamma \int_{B_R(0)}|\Delta u|^2 \, dx
+  \gamma \int_{B_R(0)} (x \cdot \nabla (\Delta u)) \Delta u  \, dx \\
&=2 \gamma \int_{B_R(0)}|\Delta u|^2 \, dx
+ \frac {\gamma}{2}\int_{B_R(0)} x \cdot \nabla (|\Delta u|^2)  \, dx \\
&= \frac{(4- N) \gamma}{2} \int_{B_R(0)}|\Delta u|^2 \, dx + \frac {\gamma}{2}\int_{ \partial B_R(0)}(x \cdot {\bf {n}}) |\Delta u|^2 \, dS. 
\end{align*}
Combining the previous two equalities, we obtain that
\begin{align*}
\gamma \int_{B_R(0)}   (x \cdot \nabla u)  \Delta^2 u \,  & dx = \frac{(4- N) \gamma}{2} \int_{B_R(0)}|\Delta u|^2 \, dx + \frac {\gamma}{2}\int_{ \partial B_R(0)}(x \cdot {\bf {n}}) |\Delta u|^2 \, dS\\
& - \gamma \int_{\partial B_R(0)} (\nabla (x \cdot \nabla u) \cdot {\bf {n}}) \Delta u  -(\nabla (\Delta u) \cdot {\bf {n}})(x\cdot \nabla u) \, dS.
\end{align*}
Next, we deal with the second left-hand side term of \eqref{inte}. We have that
\begin{align*}
- \mu \int_{B_R(0)}  (x \cdot \nabla u)&  \Delta u \,    dx = \mu \int_{B_R(0)} \nabla (x \cdot \nabla u) \cdot \nabla u \, dx
-  \mu \int_{\partial B_R(0)} (\nabla u \cdot {\bf {n}})(x \cdot \nabla u) \, dS \\
%&= \mu \int_{B_R(0)}|\nabla u|^2 \, dx + \frac{\mu}{2} \int_{B_R(0)}x \cdot \nabla (|\nabla u|^2) \, dx -  \mu \int_{\partial %B_R(0)} (\nabla u \cdot {\bf {n}})(x \cdot \nabla u) \, dS \\
&= \frac{(2-N)\mu}{2} \int_{B_R(0)} |\nabla u|^2 \, dx + \frac{\mu}{2} \int_{\partial B_R(0)}(x \cdot {\bf {n}}) |\nabla u|^2 \, dS\\
&\ \ \ \ \ \ \ \ \ \ -  \mu \int_{\partial B_R(0)} ( \nabla u \cdot {\bf {n}})(x \cdot  \nabla u) \, dS.
\end{align*}
Finally, for the last two terms of \eqref{inte}, we get
\begin{align*}
 \omega \int_{B_R(0)} (x \cdot \nabla u)u \, dx
%&= \frac { \omega}{2} \int_{B_R(0)} x \cdot \nabla (|u|^2) \, dx\\
&= -\frac{ \omega N}{2} \int_{B_R(0)} |u|^2 \,dx
+  \frac{ \omega}{2} \int_{\partial B_R(0)} (x \cdot {\bf {n}}) |u|^2 \, dS, 
\end{align*}
and
\begin{align*}
\int_{B_R(0)} (x \cdot \nabla u)|u|^{2 \sigma } u \, dx
%&= \frac{1}{2 \sigma +2}\int_{B_R(0)} x \cdot \nabla (|u|^{2 \sigma +2}) \, dx \\
&=-\frac{N }{2 \sigma +2} \int_{B_R(0)}|u|^{2 \sigma +2} \, dx + \frac{1}{2 \sigma +2} \int_{ \partial B_R(0)}(x \cdot {\bf {n}}) |u|^{2 \sigma +2} \, dS.
\end{align*}
Taking into account the above calculations, it follows from \eqref{inte} that
\begin{align} \label{final}
\begin{split}
\frac{(N-4)\gamma}{2} \int_{B_R(0)}|\Delta u|^2 \, dx &+ \frac{(N-2)\mu}{2}  \int_{B_R(0)}|\nabla u|^2 \, dx
+ \frac{N \omega}{2} \int_{B_R(0)}|u|^2 \, dx \\
%& \ \ \ \ \ \ = \frac{N }{2 \sigma + 2} \int_{B_R(0)}|u|^{2 \sigma +2} \, dx + \frac{\gamma}{R} \int_{\partial B_R(0)}  (\nabla %(\Delta u) \cdot x)(x\cdot \nabla u)- (\nabla (x \cdot \nabla u) \cdot x) \Delta u \, dS  \\
%&+ \frac {\gamma R}{2}\int_{ \partial B_R(0)} |\Delta u|^2 \, dS + \frac{\mu R}{2} \int_{  \partial B_R(0)}|\nabla u|^2 \, dS - %\frac{\mu}{R} \int_{\partial B_R(0)}|x \cdot  \nabla u|^2 \, dS \\
%& +  \frac{ \omega R}{2} \int_{\partial B_R(0)} |u|^2 \, dS - \frac{ R}{2 \sigma +2} \int_{ \partial B_R(0)} |u|^{2 \sigma +2} \, %dS\\
&   = \frac{N }{2 \sigma + 2} \int_{B_R(0)}|u|^{2 \sigma +2} \, dx +I_R(u),
\end{split}
\end{align}
where
\begin{align*}
&I_R (u)= \frac{  R}{2}\int_{ \partial B_R(0)} \left( \gamma |\Delta u|^2 +\mu |\nabla u|^2 +\omega |u|^2 - \dfrac{|u|^{2\sigma +2}}{\sigma +1}    \right) \, dS \\
&\ \ \ +\frac{1}{R} \int_{\partial B_R(0)} \left( \gamma (\nabla (\Delta u) \cdot x)(x\cdot \nabla u)-\gamma (\nabla (x \cdot \nabla u) \cdot x) \Delta u  - \mu |x \cdot  \nabla u|^2 \right) \, dS. %  - \frac{\mu}{R} \int_{\partial B_R(0)}|x \cdot  \nabla u|^2 \, dS .
%& +  \frac{ \omega R}{2} \int_{\partial B_R(0)} |u|^2 \, dS - \frac{ R}{2 \sigma +2} \int_{ \partial B_R(0)} |u|^{2 \sigma +2} \, dS.
\end{align*}
We now show that $ I_{R_n}(u) \to 0$ for a suitable sequence $(R_n)_n \subset \R$ with $R_n \to \infty$ as $n \to \infty$. First, using the Cauchy-Schwarz's inequality, we have, for any $x \in \partial{B_R(0)}$,
\begin{align} \label{bdy}
\begin{split}
\left|(\nabla (\Delta u) \cdot x)(x \cdot u)\right| &\leq R^2 \left(|\nabla (\Delta u)|^2 + | u|^2 \right) \\
|(\nabla (x \cdot \nabla u) \cdot x) \Delta u| & \leq  C_N R^2\left( |\Delta u|^2 + \sum_{i, j=1}^N |u_{i,j}|^2 + |\nabla u|^2\right), \\
\end{split}
\end{align}
where $u_{i, j}:= \frac{\partial^2 u}{\partial x_i \partial x_j}$. In view of the elliptic regularity theory, we have that $u \in H^4(\R^N)$, in particular $u \in H^3(\R^N)$. 
This yields to
\begin{align} \label{contra}
\begin{split}
&\int_{\R^N}|\nabla (\Delta u)|^2 + |\Delta u|^2 +  \sum_{i, j=1}^N |u_{i,j}|^2 + |\nabla u|^2 + |u|^2 + |u|^{2 \sigma +2} \, dx \\
&= \int_{0}^{\infty} \Big(\int_{\partial {B_R(0)}} |\nabla (\Delta u)|^2 + |\Delta u|^2 +  \sum_{i, j=1}^N |u_{i,j}|^2 + |\nabla u|^2 + |u|^2 + |u|^{2 \sigma +2} \, dS \Big) \,dR < \infty.
\end{split}
\end{align}
As a consequence, there exists a sequence $(R_n)_n \subset \R^N$ satisfying $R_n \to \infty$ as $n \to \infty$ so that
$$
R_n\int_{\partial {B_{R_n}(0)}} |\nabla (\Delta u)|^2 + |\Delta u|^2 +  \sum_{i, j=1}^N |u_{i,j}|^2 + |\nabla u|^2 + |u|^2 + |u|^{2 \sigma +2} \, dS \to 0.
$$
This implies that $I_{R_n}(u) \to 0$ as $n \to \infty$. Now substituting $R$ by $R_n$ in \eqref{final}, we then obtain that $P_{\omega}(u)=0$. This completes the proof.
\end{proof}

\begin{remark} \label{positive}
Let $4 \leq \sigma N <4^*$. From Lemma \ref{identities}, any ground state solution $u \in H^2(\R^N)$ to (\ref{4nls}) satisfies $Q(u)=0$. Thus
$$
E_0(u) = E_0(u) - \frac{2}{\sigma N}Q(u) =  \frac{(\sigma N -4)\gamma}{2 \sigma N} ||\Delta u||_2^2 + \frac{(\sigma N -2) \mu}{2 \sigma N} ||\nabla u||_2^2
$$
and this shows that $E_0(u) > 0$ (unless in the particular case $\sigma N =4$ and $\mu =0$ where $E_0(u)=0$).
\end{remark}

\subsection{Variational characterization of the ground state energy level}

Here we derive, in Proposition \ref{propmin}, an alternative characterization of $d_{\omega}$ that is central in our proof of Theorem \ref{mainthm1}. \medskip 

For any given $u \in H^2(\R^N)$, let $u_{\lambda}(x)=\lambda^{\frac{N}{4}}u(\sqrt{\lambda} x)$ for $\lambda>0$.

\begin{lem}\label{unique}
Let $u\in H^2(\R^N)$ be such that $Q(u)\leq 0$. Then there exists a unique $\lambda_0\in (0,1]$ such that the followings hold:
\begin{enumerate}
\item $Q(u_{\lambda_0})=0$; \smallskip
\item $\lambda_0=1$ if and only if $Q(u)=0$; \smallskip
\item  $E_{\omega}(u_\lambda)< E_{\omega}(u_{\lambda_0})$, for any $\lambda>0$, $\lambda \neq \lambda_0$; \smallskip
\item $ \lambda \mapsto E_{\omega}(u_{\lambda})$ is concave on $[\lambda_0, \infty)$.
\end{enumerate}
\end{lem}

\begin{proof}
Making a change of variables, we have
\begin{equation*}
E_{\omega}(u_{\lambda})=\frac{\gamma\lambda^{2}}{2}\|\Delta u\|_2^2+\frac{\lambda}{2}\mu\|\nabla u\|_2^2+\dfrac{\omega}{2} \|u\|_2^2 -\frac{\lambda^{\sigma N/2}}{2\sigma+2}\|u\|_{2\sigma+2}^{2\sigma+2},
\end{equation*}
and differentiating with respect to $\lambda$, we find
\begin{align*}
\dfrac{d}{d\lambda}E_{\omega}(u_{\lambda})&=\gamma\lambda\|\Delta u\|_2^2+\frac{\mu}{2}\|\nabla u\|_2^2-\frac{\sigma N\lambda^{\sigma N/2 -1}}{2(2\sigma+2)}\|u\|_{2\sigma +2}^{2\sigma+2}\\
&= \dfrac{1}{\lambda}Q(u_\lambda)
\end{align*}
and
\begin{align} \label{second}
\dfrac{d^2}{d^2\lambda}E_{\omega}(u_{\lambda})=\gamma \|\Delta u\|_2^2 - \frac{\sigma N (\sigma N -2)}{4 (2 \sigma +2)} t^{\frac{\sigma N}{2}-2} \lambda_0^{\frac{\sigma N}{2}-2} \|u\|_{2\sigma+2}^{2\sigma+2}.
\end{align}
Since $\sigma N \geq 4$, and $Q(u) \leq 0$, it is easily seen that there exists a unique $\lambda_0 \in (0,1]$ such that
$$
Q(u_{\lambda_0}) = \lambda_0 \frac{d}{d \lambda}E_{\omega}(u_{\lambda})|_{\lambda = \lambda_0} =0,
$$
and also that
$$
\dfrac{d}{d\lambda}E_{\omega}(u_\lambda )=\begin{cases}>0,\ \mbox{ if } \lambda \in (0,\lambda_0),\\ <0,\ \mbox{ if } \lambda\in (\lambda_0 , \infty ),\end{cases}
$$
from which we get that $E_{\omega}(u_\lambda)< E_{\omega}(u_{\lambda_0})$ for any $\lambda>0$, $\lambda \neq \lambda_0$. Thus $(1)$-$(3)$ hold. Now writing $\lambda = t \lambda_0$ for $t\geq 1$, in view of \eqref{second}, we have
\begin{align*}
\dfrac{d^2}{d^2\lambda}E_{\omega}(u_{\lambda})= \dfrac{1}{\lambda_0^2}\left(\gamma \lambda_0^2 \|\Delta u\|^2_2 - \frac{\sigma N (\sigma N -2)}{4 (2 \sigma +2)} t^{\frac{\sigma N}{2}-2} \lambda_0^{\frac{\sigma N}{2}} \|u\|_{2\sigma +2}^{2\sigma+2}\right).
\end{align*}
Because of $\sigma N \geq 4$, and $t \geq 1$, using the fact that $Q(u_{\lambda_0})=0$, that is
$$
\gamma \lambda_0^2 \|\Delta u\|_2^2 + \frac{\mu}{2} \lambda_0 \|\nabla u\|_2^2 - \frac{\sigma N }{2 (2 \sigma +2)} \lambda^{\frac{\sigma N}{2}} \|u\|_{2\sigma +2}^{2\sigma+2}=0,
$$
then there holds
$$
\frac{d^2}{d^2 \lambda} E_{\omega}(u_{\lambda}) <0 \ \text{for}\ \lambda \geq \lambda_0,
$$
and this ends the proof.
\end{proof}

It is standard to show, if \eqref{4nls} admits a ground state solution, that
\begin{align} \label{ground}
d_{\omega} = \inf \{E_{\omega}(u) \ | \ u\in H^2(\R^N) \backslash \{0\},\ I_{\omega}(u)=0  \},
\end{align}
and that the infimum in (\ref{ground}) is achieved. Using this property we now obtain the following result.

\begin{prop}\label{propmin}
Let
$$
c_\omega := \inf \{E (u) \ | \ u\in \mathcal{M}_{\omega}\}
$$
where $\mathcal{M}_{\omega}:= \{u\in H^2 (\R^N)\backslash \{0\} \ | \ Q(u)=0,\ I_{\omega}(u)\leq 0 \}.$
Then $d_{\omega} \leq c_\omega$ and the equality holds if $d_{\omega}$ is achieved.
\end{prop}

\begin{proof}
Let $u\in \mathcal{M}$. We can assume without loss of generality that $I_{\omega} (u)<0$. Otherwise, if $I_{\omega}(u)=0$, we  have $E_{\omega}(u)\geq d_{\omega}$. Because $\sigma N \geq 4$, we observe that
$\lim_{\lambda \rightarrow 0^+}I_{\omega}( u_\lambda)>0$. Therefore, by continuity, there exists a $\hat\lambda \in (0,1)$ such that $I_{\omega}(u_{\hat \lambda} )=0$, and $E_{\omega}(u_{\hat\lambda})\geq d_{\omega}$. Using that $Q(u)=0$, we deduce from Lemma \ref{unique} that $E_{\omega}(u) \geq E(u_{\hat \lambda} )\geq d_{\omega}$. Thus we obtain  $c_{\omega} \geq d_{\omega}$. Next we assume that $d_{\omega}$ is achieved by some $v \in H^2(\R^N)$. Since $v$ is a solution to \eqref{4nls} it follows from Lemma \ref{identities} that $I_{\omega}(v)=0$ and $Q(v)=0$. Hence  $ v \in \mathcal{M}_{\omega}$ and  $c_\omega \leq E_{\omega}(v)=d_{\omega}$.
\end{proof}

\subsection{The blow-up alternative and a localized virial identity}

We begin by recalling the local well-posedness to the Cauchy problem \eqref{4nlsdis} and a blow-up alternative due to \cite[Proposition 4.1]{Pa}.

\begin{lem} \label{exloc}
Let $\sigma N < 4^*$. For any $u_0\in H^2 (\R^N)$, there exist $T>0$ and a unique solution $u(t) \in C([0,T); H^2 (\R^N))$ to \eqref{4nlsdis} with the initial datum $u_0$ so that the mass and the energy are conserved along the time, that is for any $t \in [0, T)$
$$
\|u(t)\|_{2}=\|u_0\|_{2}, \ E_{\omega}(u(t))=E_{\omega}(u_0).
$$
Moreover, we have the alternative that either $T=\infty,$ or $\lim_{t\rightarrow T^-}\|\Delta u(t) \|_2=\infty.$
\end{lem}

Next we recall the localized virial to \eqref{4nlsdis} which has been introduced by Boulenger and Lenzmann \cite{BoLe}. This quantity will play a crucial role to deduce the occurrence of  blow-up. Let $\varphi : \R^N \rightarrow \R$ be a radial function such that $D^j \varphi \in L^\infty (\R^N)$, $1\leq j \leq 6$,
$$
\varphi (r):=\begin{cases}\dfrac{r^2}{2} & for\ r\leq 1 \\ const. & for\ r\geq 10  \end{cases}, and\ \varphi'' (r)\leq 1,\ for\ r\geq 0.
$$
Let $R>0$, we set $\varphi_R (r):= R^2 \varphi (\dfrac{r}{R})$. For $u\in H^2 (\R^N)$, we defined the localized virial $M_{\varphi_R}$ by
\begin{equation}\label{localvirial}
M_{\varphi_R} [u]:= 2 \text{Im} \int_{\R^N}\bar{u} \nabla \varphi_R \nabla u\ dx.
\end{equation}
The following lemma reveals a key information on the evolution of this quantity.
\begin{lem} \label{locvir} \cite[Lemma 3.1]{BoLe}
Let $\sigma N < 4^*, N\geq 2$, and $R>0$. Suppose that $u(t) \in C([0,T); H^2_{rad} (\R^N))$ is a solution to \eqref{4nlsdis}. Then for any $t\in [0,T)$,
\begin{align}\label{blevt} \nonumber
\dfrac{d}{dt}M_{\varphi_R} [ u (t) ] &\leq 4 N\sigma E_0(u(t)) - (2N\sigma -8)\gamma \|\Delta u(t)\|_2^2 - (2N\sigma -4 )\mu \|\nabla u (t)\|_2^2 \\
&  +O\left(\frac{1}{R^4}+ \dfrac{\|\nabla u (t)\|_2^2}{R^2}+ \dfrac{\|\nabla u(t)\|_2^\sigma}{R^{\sigma (N-1)}}+\dfrac{\mu}{R^2} \right)\\ \nonumber
& =  8 Q(u(t)) + O\left(\frac{1}{R^4} + \dfrac{\|\nabla u(t)\|_2^2}{R^2}+ \dfrac{\|\nabla u(t)\|_2^\sigma}{R^{\sigma (N-1)}}+\dfrac{\mu}{R^2} \right).
\end{align}
Moreover, if $\mu=0$ and $\sigma N=4$, we have (see $(7.12)$ and $(7.13)$ of \cite{BoLe})
\begin{equation}
\label{virlimitcase}
\dfrac{d}{dt}M_{\varphi_R} [u(t)]\leq 8Q (u(t))+0(\dfrac{1}{\eta R^2} +\eta^{1/2}),
\end{equation}
for, $R\geq 1$ and $0<\eta <1$.
\end{lem}

\section{Proof of Theorems \ref{mainthm1} and \ref{mainthm11}.}\label{sec:proof}

In this last section we give the proof of our main results.

\begin{proof}[Proof of Theorem \ref{mainthm1}]
We first consider the case $\mu\ne 0$. Let $\varepsilon>0$ be fixed, and $u \in H^2_{rad} (\R^N)$ be a ground state solution to \eqref{4nls}. Since $Q(u)=0$, thanks to Lemma \ref{unique}, for $\lambda >1$ sufficiently close to $1$, we have
\begin{equation}\label{thme1}
E_{\omega}(u_{\lambda} )< E_{\omega}(u), \ Q(u_{\lambda})<0.
\end{equation}
Also notice that $\|u_{\lambda} -u \|_{H^2}\leq \varepsilon$, and
\begin{align}\label{thme2}
\begin{split}
I_{\omega}(u_{\lambda} )&=2 E_{\omega}(u_{\lambda} ) +\dfrac{4}{N} \left(Q(u_{\lambda}) - \gamma \|\Delta u_{\lambda}\|_2^2 -\frac{\mu }{2}\|\nabla u_{\lambda} \|^2_2  \right) \\
&\leq 2 E_{\omega}(u) -I_{\omega}(u) +\dfrac{4}{N} \left( Q(u) -  \lambda^2 \gamma \|\Delta u\|_2^2-\frac{\lambda\mu}{2} \|\nabla u \|_2^2\right)\\
&=\dfrac{4}{N} \left((1-\lambda^2) \gamma \|\Delta u\|_2^2 +\frac{(1-\lambda)\mu}{2}\|\nabla u \|_2^2 \right)<0.
\end{split}
\end{align}
We set $v:=u_{\lambda}$. In view of Lemma \ref{exloc}, we know that there exists a unique solution $\phi(t) \in C([0, T); H^2_{rad}(\R^N))$ to \eqref{4nlsdis} with initial datum $\phi(0)=v$, where $T>0$ is the maximum existence time to $\phi(t)$. At this point, in view of Definition \ref{unstable},  to prove that $u$ is unstable, we must show that the solution $\phi(t)$ blows up in finite time. We divide the rest of the proof into four steps. \medskip \\
{\bf First step :} We claim that for any $t\in [0,T)$, there holds
\begin{equation} \label{thme4}
 E_{\omega}(\phi (t))< E_{\omega}(u), \ I_{\omega}(\phi (t))<0, \ \text{and} \ Q(\phi (t))< 0.
\end{equation}
By Lemma \ref{exloc}, and \eqref{thme1}, we get for any $t\in [0,T)$,
\begin{equation} \label{thme3}
E_{\omega}(\phi (t))=E_{\omega}(v)< E_{\omega}(u).
\end{equation}
We now show that $I_{\omega}(\phi(t))<0$ for any $t \in [0, T)$. We assume by contradiction  that there is $t_0 \in [0, T)$ so that $I_{\omega}(\phi(t_0)) =0$. Thus by \eqref{ground}, and \eqref{thme1}, we obtain that $d_{\omega} \leq E_{\omega}(\phi(t_0))=E(v) < E_{\omega}(u)$, which contradicts the fact that $E_{\omega}(u)=d_{\omega}$. From \eqref{thme2}, we have that $I_{\omega}(\phi(0)) <0$, and thus $I_{\omega}(\phi(t))<0$ for any $t \in [0, T)$. Finally we show that $Q (\phi(t))<0$ for any $t \in [0, T)$. We suppose that $Q(\phi(t_1))=0$ for some $t_1 \in [0, T)$. Since $I_{\omega}(\phi(t_1)) <0$, we obtain that $\phi(t_1) \in \mathcal{M}_{\omega}$. Thanks to Proposition \ref{propmin}, we then deduce that $E_{\omega}(\phi(t_1))\geq E_{\omega}(u)$, which contradicts \eqref{thme3}. Since $Q(\phi (0))<0$, then $Q(\phi (t))<0$ for any $t\in [0,T)$. This establishes \eqref{thme4}. \medskip \\
{\bf Second step :} We claim that there exists a constant $a >0$ (not depending on $t$) such that $Q(\phi(t)) \leq -a$ for any $t\in [0,T)$. Let $t \in [0, T)$ be arbitrary but fixed. In view of \eqref{thme4}, $Q(\phi(t))<0$ and hence by Lemma \ref{unique} there exists a unique $\lambda_0 \in (0,1)$ such that $Q(\phi(t)_{\lambda_0}) =0$. If $I_{\omega}(\phi(t)_{\lambda_0}) \leq 0$, we define $\lambda^*:=\lambda_0$, otherwise we take $\lambda^*>1$ such that $I_{\omega}(\phi(t)_{\lambda^*}) = 0$, and $Q(\phi(t)_{\lambda^*}) <0$. In any case we obtain that $E_{\omega}(\phi(t)_{\lambda^*}) \geq d_{\omega}$, and $Q(\phi(t)_{\lambda^*}) \leq 0$. This, together with that fact from Lemma \ref{unique} that $\lambda \mapsto E_{\omega}(u_{\lambda})$ is concave on $[\lambda_0, \infty)$, gives 
\begin{align*} \label{locvirial2}
\begin{split}
E_{\omega}(v) = E_{\omega}(\phi(t)) & \geq  E_{\omega}(\phi(t)_{\lambda^*}) + (1 - \lambda^*) \frac{d}{d \lambda} E_{\omega}(\phi(t)_{\lambda})|_{\lambda =1}\\
& \geq d_{\omega} + (1 - \lambda^*) Q(\phi(t)).
\end{split}
\end{align*}
This yields that $Q(\phi(t)) \leq - (d_{\omega}-E_{\omega}(v)):=-a$. Noting that $a>0$ because of \eqref{thme1} this proves the claim. \medskip \\
{\bf Third step :} Let us show that there exist a constant $\delta>0$ such that
\begin{equation}\label{thmsecondstep}
\dfrac{d}{dt}M_{\varphi_R} [\phi (t)]\leq - \delta \|\nabla \phi (t)\|_2^2 \ \text{for} \ t \in [0,T),
\end{equation}
and a $t_1 \geq 0$ such that
\begin{equation}\label{initial}
M_{\varphi_R} [\phi (t)]<0  \ \text{for} \ t \geq t_1.
\end{equation}
Observe that $\phi (t)$ is radial for any $t\in [0,T)$, since the initial datum $\phi(0)=v$ is radial. Therefore, to prove
\eqref{thmsecondstep}, we can apply Lemma \ref{locvir}. To this aim, we now distinguish two cases.  \medskip \\
{\it Case 1 :}  Assume that $t \in [0, T)$ is such that
\begin{equation}
\label{thmdect}
||\nabla \phi (t)||_2^2 \leq \frac{4N \sigma E_0(v)}{\mu (N \sigma -2)}.
\end{equation}
Since $N \geq 2$, taking $R>0$ large enough,  we can insure that
$$
O\left(\frac{1}{R^4}+ \dfrac{\|\nabla \phi (t)\|_2^2}{R^2}+ \dfrac{\|\nabla \phi (t)\|_2^\sigma}{R^{\sigma (N-1)}}+\dfrac{\mu}{R^2} \right)< a,
$$
where $a>0$ is the constant determined in the Second step.
Recalling that $Q(\phi(t)) \leq - a$ for any $t \in [0, T)$, we deduce from Lemma \ref{locvir} that
\begin{equation}\label{esti1}
\dfrac{d}{dt}M_{\varphi_R} [ \phi (t) ] \leq - 7 a \leq  - \delta ||\nabla \phi (t)||_2^2
\end{equation}
for some $\delta >0$ sufficiently small. \medskip \\
{\it Case 2 :}  Assume that $t \in [0, T)$ is such that
$$||\nabla \phi (t)||_2^2 > \frac{4N \sigma E_0(v)}{\mu (N \sigma -2) }.$$
Since $||\nabla \phi (t)||_2^2 \leq ||\phi (t)||_2 ||\Delta \phi (t)||_2$, using \eqref{blevt}, we have that
\begin{align*} \label{locvirial2}
\begin{split}
\dfrac{d}{dt}M_{\varphi_R} [ \phi (t) ] &\leq 4 N\sigma E_0(v) - \frac{(2N\sigma -8)\gamma}{||\phi (t)||_2^2} \|\nabla \phi (t)\|_2^4  - (2N \sigma -4)\mu ||\nabla \phi(t)||_2^2\\
& +   O\left(\frac{1}{R^{4}}+ \dfrac{\|\nabla \phi (t)\|_2^2}{R^2}+ \dfrac{\|\nabla \phi (t)\|_2^\sigma}{R^{\sigma (N-1)}}+\dfrac{\mu}{R^2} \right) \\
&\leq  - (N\sigma -2)\mu \|\nabla \phi (t)\|_2^2 - \frac{(2N\sigma -8)\gamma}{||\phi (0)||_2^2} \|\nabla \phi (t)\|_2^4   \\
& + O\left(\frac {1}{R^{4}}+ \dfrac{\|\nabla \phi (t)\|_2^2}{R^2}+ \dfrac{\|\nabla \phi (t)\|_2^\sigma}{R^{\sigma (N-1)}}+\dfrac{\mu}{R^2} \right) .
\end{split}
\end{align*}
Taking $R$ large enough and using our assumptions that $\sigma \leq 2$ if $N \sigma =4$ and $\sigma \leq 4$ if $\sigma N >4$ we deduce that
\begin{equation} \label{esti2}
\dfrac{d}{dt}M_{\varphi_R} [ \phi (t) ]  \leq -\dfrac{ (N \sigma  -2)\mu}{2}   \|\nabla \phi (t)\|_2^2.
\end{equation}
Combining the estimates \eqref{esti1}-\eqref{esti2}, we see that there exists a $\delta>0$ such that (\ref{thmsecondstep}) holds. Now, since
$$ M_{\varphi_R} [ \phi (t) ] = M_{\varphi_R} [ \phi (0) ] + \int_{0}^{t_1} \dfrac{d}{ds}M_{\varphi_R} [ \phi (s) ] ds,$$
the inequality \eqref{initial} follows
from the estimate
$$ \big| \dfrac{d}{dt}M_{\varphi_R} [ \phi (t) ] \big| \geq \min \left\{ 7a, \frac{(N \sigma -2)}{2}||\nabla \phi(t)||_2 \right\}.$$
\medskip \\
{\bf Fourth step :} We now conclude that the solution $\phi(t)$ to \eqref{4nlsdis} with the initial datum $\phi(0)=v$ blows up. Here we adapt another argument from \cite{BoLe}. Suppose by contradiction that $T=\infty$, then integrating \eqref{thmsecondstep} on $[t_1 ,t]$, and taking into consideration \eqref{initial}, we have that
$$
M_{\varphi_R} [ \phi (t) ] \leq  -\delta \int_{t_1}^t \|\nabla \phi (s)\|_2^2 ds.\
$$
Now using Cauchy-Schwarz's inequality, we get from the definition of $M_{\varphi_R} [ \phi (t) ] $ given by \eqref{localvirial} that
$$
|M_{\varphi_R}[\phi (t)] | \leq 2 \|\nabla \varphi_R\|_{\infty} \|\phi(t)\|_2 \|\nabla \phi(t)\|_2 \leq C \|\nabla \phi (t)\|_2.
$$
Thus for some $\tau >0$,
\begin{equation}\label{control}
M_{\varphi_R}[\phi (t)] \leq - \tau \int_{t_1}^t | M_{\varphi_R}[\phi (s)] |^2 ds.
\end{equation}
Setting $z(t):=\int_{t_1}^t |M_{\varphi_R}[\phi (s)] |^2 ds $, we obtain from \eqref{control} that
$$
z'(t)\geq \tau^2 z(t)^2.
$$
Integrating this differential inequality on $[t_1, t]$, we get
$$
z(t) \geq \frac{z(t_1)}{1-\tau^2(t-t_1)z(t_1)},
$$
and thus $M_{\varphi_R}[\phi (t)]\rightarrow -\infty$ as $t$ tends to some finite time $t^\ast$. Therefore the solution $\phi(t)$ cannot exist globally and by the blow-up alternative recalled in Lemma \ref{exloc} this ends the proof in the case $\mu\ne 0$.

\medbreak

Assume now that $\mu = 0$. In this case, the first two steps remain unchanged and proceeding as in the third step, one can show that  there exist a constant $\delta>0$ such that
\begin{equation}\label{thmsecondstepmu}
\dfrac{d}{dt}M_{\varphi_R} [\phi (t)]\leq - \delta \|\Delta \phi (t)\|_2^2 \ \text{for} \ t \in [0,T),
\end{equation}
and a $t_1 \geq 0$ such that
\begin{equation}\label{initialmu}
M_{\varphi_R} [\phi (t)]<0  \ \text{for} \ t \geq t_1.
\end{equation}
Indeed, the conclusion follows if we substitute \eqref{thmdect} by
$$||\Delta \phi (t)||_2^2 \leq \frac{4N \sigma E(v)}{(N \sigma -4) \gamma}.$$
Now, integrating \eqref{thmsecondstepmu} on $[t_1 ,t]$ and using \eqref{initialmu}, we have
$$M_{\varphi_R} [ \phi (t) ] \leq  -\delta \int_{t_1}^t \|\Delta \phi (s)\|_2^2 dx .$$
Using Cauchy-Schwarz inequality, we get
$$|M_{\varphi_R}[\phi (t)] |^4 \leq C \|\Delta \phi (t)\|_2^2,$$
for some constant $C$ depending on $R$ and $v$. Combining the previous inequalities and setting $z(t)= \int_{t_1}^t |M_{\varphi_R}[\phi (s)] |^4 ds $, we obtain
$$z ' (t) \geq C z (t)^4.$$
Integrating this equation, we deduce that $M_{\varphi_R}[\phi (t)]\rightarrow -\infty$ when $t$ tends to some finite time $t^\ast$. This concludes the proof in the case $\mu=0$ and $\sigma N > 4$. 
\end{proof}

\medbreak

\begin{proof}[Proof of Theorem \ref{mainthm11}]

We first consider the case $\mu=0$ and $\sigma N=4$. Since $Q(\phi (t))<-a$, we can choose $\eta>0$ sufficiently small and $R\geq 1$ sufficiently large to deduce from \eqref{virlimitcase} that
\begin{equation}
\label{lasteq}
\dfrac{d}{dt}M_{\varphi_R} [\phi (t)]\leq -4 a.
\end{equation}
Proceeding as in \cite{BoLe}, suppose by contradiction that $T=\infty$. Then, there exists $t_1\geq 0$ such that 
$M_{\varphi_R} [\phi (t)] <0$, for all $t\geq t_1$. Using the Cauchy-Schwarz's inequality and integrating \eqref{lasteq} between $[t_1 ,t]$, we find
$$-\|\nabla \varphi_R \|_{\infty} \|\phi (t) \|_2^{3/2} \|\Delta \phi (t)\|_2^{1/2}\leq M_{\varphi_R}[\phi (t)]\leq - 4a (t-t_1).$$
Therefore, we see that either $\phi (t)$ blows up in finite time or that
$$\|\Delta \phi (t)\|_2^2 \geq C (t-t_1)^2,\ for\ all\ t\geq t_1 .$$

Consider now the case $\sigma > 4$. If $\|\nabla \phi (t)\|_2$ is unbounded, then $\|\Delta \phi (t)\|_2$ is unbounded and $\phi(t)$ blows up either in finite or infinite time. If $\|\nabla \phi (t)\|_2$ is bounded, then \eqref{blevt} implies that  
\eqref{lasteq} holds for $R\geq 1$ sufficiently large and we conclude as above. 
\end{proof}

%\begin{remark}
%It would be interesting to remove the assumption $\sigma \leq 4$ (in small dimension) from the statement of Theorem \ref{mainthm1}. We leave that as an open question. However, we can prove a weaker statement. Namely, without this assumption $\sigma \leq 4$
%\end{remark}

\end{document}